\documentclass[10pt, a4paper, reqno]{amsart}

\usepackage[utf8]{inputenc}

\usepackage[english]{babel}

\usepackage{amsmath}
\usepackage{amsfonts}
\usepackage{amssymb}
\usepackage{amsthm}
\usepackage[foot]{amsaddr}
\usepackage{mathtools}
\usepackage{bm, bbm}
\usepackage{bigints}
\usepackage{hyperref}
\usepackage[capitalise]{cleveref}
\usepackage{xparse}
\usepackage{graphicx}
\usepackage[percent]{overpic}
\usepackage[usenames,dvipsnames,svgnames,table]{xcolor}
\usepackage[scr=boondoxo, scrscaled=.98]{mathalfa}
\usepackage{dsfont}

\usepackage{enumerate
}

\usepackage[numbers]{natbib}

\nocite{*}

\usepackage[left=3cm, right=3cm, top=3.5cm, bottom=4cm]{geometry}

\DeclareDocumentCommand{\bmd}{O{s} O{y} O{0} O{t} O{u}}
{P \Big\{ B(#1) \in \mathrm d #2 \,\Big \vert \min_{#3\leq z \leq #4} B(z)> 0 , B(#3)=#5 \Big\}}

\DeclareDocumentCommand{\bmdd}{O{s} O{y} O{0} O{t} O{u}}
{P \Big\{ B^\mu(#1) \in \mathrm d #2 \,\Big \vert \min_{#3\leq z \leq #4} B^\mu(z)> 0 , B^\mu(#3)=#5 \Big\}}

\DeclareDocumentCommand{\bmdv}{O{s} O{y} O{0} O{t} O{u}}
{P \Big\{B(#1) \in \mathrm d #2 \,\Big \vert \min_{#3\leq z \leq #4} B(z)> v , B(#3)=#5 \Big\}}

\DeclareDocumentCommand{\bmdb}{O{s} O{y} O{0} O{t} O{u} O{v} }
{P \Big\{B(#1) \in \mathrm d #2 \,\Big \vert \min_{#3\leq z \leq #4} B(z)> 0 , B(#3)=#5, 
	B(#4)=#6 \Big\}}

\DeclareDocumentCommand{\bmdbd}{O{s} O{y} O{0} O{t} O{u} O{v} }
{P \Big\{B^\mu(#1) \in \mathrm d #2 \,\Big \vert \min_{#3\leq z \leq #4} B^\mu(z)> 0 , 	
	B^\mu(#3)=#5, 
	B^\mu(#4)=#6 \Big\}}

\newcommand\numberthis{\addtocounter{equation}{1}\tag{\theequation}}
\numberwithin{equation}{section}

\allowdisplaybreaks

\theoremstyle{remark}
\newtheorem{remark}{Remark}[section]

\theoremstyle{plain}
\newtheorem{theorem}{Theorem}[section]
\newtheorem{lemma}{Lemma}[section]

\theoremstyle{definition}

\theoremstyle{definition}
\newtheorem{definition}{Definition}[section]

\title{On the sojourn time of a Generalized Brownian meander}
\author{F. Iafrate$^1$ and E. Orsingher$^1$}
\address{$^1$Department of Statistical Sciences, Sapienza University of Rome, Italy.}
\email{enzo.orsingher@uniroma1.it, francesco.iafrate@uniroma1.it}
\date{\scriptsize \texttt{\today}}

\begin{document}

\begin{abstract}
In this paper we study the 
sojourn time on the positive half-line up to time $ t $ of a drifted Brownian motion
 with starting point $ u $
and subject to the condition that $ \min_{ 0\leq z \leq l} B(z)> v  $, with $  u > v $.
This  process is a drifted Brownian meander up to time $ l $ and then evolves as a free Brownian motion. 
We also consider the sojourn time of a bridge-type process, where we add the  additional condition to return to the initial level at the end of the time interval. 
We analyze the weak limit of the occupation functional as $ u \downarrow v $. We obtain explicit distributional results when the barrier is placed at the zero level, 
and also in the special case when the drift is null. 

\end{abstract}
\maketitle


\section{Introduction}

The Brownian meander is a Brownian motion 
$\{B(t)\,,\,\, t \geq 0\}$
evolving under the condition that $ \min_{0\leq s\leq t} B(s) > 0 $. 
If the additional condition that $ B(t) = c >0  $ is assumed then we have the Brownian excursion which 
is the bridge of the Brownian meander.


Early results in this field emerged in the study of the behaviour of 
random walks conditioned to stay positive where the Brownian meander 
was obtained as the weak limit of such conditional processes (\cite{belkin70}, 
\cite{iglehart74}). In the same spirit the distribution of the maximum 
of the Brownian meander and excursion has been derived in \cite{kaigh78}. 
Further investigations about such distributions have been presented in \cite{chung1976}. 
Some important results can also be found in the classical book by 
\citet{ito1996diffusion}. The notion 
of Brownian meander as a conditional Brownian motion and 
problems concerning weak convergence to such processes have been treated 
in \cite{durrett77}. 
More recently analogous results have been obtained in the general setting
of Lévy processes (see for example \cite{chaumont2005}).

Brownian meanders emerge in path decompositions of the Brownian motion.
In particular Denisov (\cite{denisov84})  shows that a Brownian motion around a maximum 
point can be represented (in law) by means of a two-sided Brownian meander, which is constructed by gluing together 
two meanders. 

These processes also arise in several scientific fields. 
Possible applications range from SPDE's with reflection (\cite{zambotti04}) to enumeration of random graphs (see \cite{janson07} for a survey of the results in this field).

In the recent paper \cite{iafrate19} the authors have  studied the drifted Brownian meander, and provided a construction of such a process 
within the general setting of weak convergence of measures (\cite{billingsley2009convergence}).

We consider below a Brownian motion with drift which alternates intervals 
where it behaves as a meander and as a free Brownian motion.
The aim of  this paper is to study the sojourn distribution on $ (0, \infty) $
of a Brownian motion with drift $ B^\mu(t), t>0 $ which is a drifted meander
up to time $ 0 < l < t $ and in $ [l,t]$ is a drifted Brownian motion with Rayleigh 
distributed starting point. 

Our aim is therefore the analysis of the r.v.
\begin{equation}\label{eq:souj-def-intro}
 \Gamma^\mu_{l,t} =  \int_{l}^{t} \mathbbm 1_{[0, \infty)}( B^\mu(s)) \mathrm d s 
\end{equation}
under the condition that
\begin{equation}
\left\{ 	\min_{0\leq s\leq l} B^\mu(s) > v, B^\mu(0) = u \right\}\,, \qquad l<t
\end{equation}

In order to obtain explicit results for \eqref{eq:souj-def-intro}
we will consider the case where $ u \downarrow v $, $ v = 0 $, $ \mu = 0 $.
In this case we have that
\begin{align}\label{eq:souj-dist-intro}
\lim\limits_{u \to 0 } 
&P\left\{  \Gamma_{l,t} \in \mathrm d s \Big |  \min_{0\leq z\leq l} B(z) > 0, B(0)=u  \right\} = \\
&=
\left\{ 
\begin{aligned}
&\frac{1}{\pi \,\sqrt{s(t-l-s)}  } \frac{s}{s + l} \, \mathrm d s
& 0 \leq s <t-l \\
& \sqrt \frac lt & s = t-l
\end{aligned}
\right. \notag 
\end{align}
Also the case with $ \mu \neq 0 $ is considered but for a non-zero 
drift the sojourn time distribution of $ \Gamma^\mu_{l,t} $ loses the arc-sine form \eqref{eq:souj-dist-intro}. 
As a byproduct of our analysis we give the distribution  
$ P\left\{ \Gamma^\mu_t \in \mathrm d s | B^\mu(0) = x \right\} $
and establish its connection with the transition function of the elastic Brownian motion.
The condition that $ \min_{0\leq s\leq l} B^\mu (s) > 0$ 
implies that there is a positive probability of never crossing the 
zero level during $ (l,t) $. 

In the last part of the paper we consider a Brownian motion which behaves as
a
meander in $ [0,l] $ and at time $ t>l $ is obliged to pass through some point $ y $. 
This is a sort of generalized Brownian excursion. 
For the generalized Brownian excursion we have  that

\begin{align}\label{eq:souj-exc-dist-intro}
&P\left\{  \Gamma^\mu_{l,t} \in \mathrm d s \Big | \inf_{0< z\leq l} B^\mu(z) > 0, B^\mu(0)=0 , B^\mu(t)=0 \right\} 
\\
& \qquad = 
\frac 1t \sqrt \frac{l } {t-l}
\left\{
\frac{t - 2l}{  \sqrt{  l(t - l) }} - 
\frac{t - 2(l + s)}{  \sqrt{  (l+s)(t - l -s) }}  
\right\} \mathrm d s  \, . \notag
\end{align}

Result \eqref{eq:souj-exc-dist-intro} does not depend on the drift 
$  \mu $ because in the Brownian excursion the condition $ B^\mu(t) = 0 $
cancels the drift effect.

A special case of \eqref{eq:souj-exc-dist-intro} is $ l = t/2 $ 
and \eqref{eq:souj-exc-dist-intro} reduces to	
\begin{align*}
& P\left\{  \Gamma^\mu_{\frac t2,t} \in \mathrm d s \Big |  \min_{0\leq z\leq l} B(z) > 0, B(0)=0 , B(t)=0 \right\}
= \\
&= \frac{ 4s}{ t \sqrt{ t^2 - 4s^2}} \mathrm d s \qquad 0<s<\frac  t2 \notag.
\end{align*}


\section{Preliminaries}\label{sec:prelim}
Let $ \{ B(t),  t \in [0,T] \}$ be a Brownian motion adapted to the natural filtration on some measurable space $ (\Omega, \mathscr F) $
and let $ \{ P_u , u \in \mathbb R \} $ be a family of probability measures such that, under each $ P_u $, $ B $ is a Brownian motion and $  P(B(0) = u) = 1$. 
We consider a drifted Brownian motion $ \{ B^\mu (t), 0 \leq t \leq T\}  $ 
defined as $ B^\mu ( t) = B(t) + \mu t \, , 0 \leq t \leq T  $, 
with $ \mu \in \mathbbm{R} $.
The space $ C[0,T] $ of its sample paths, sometimes indicated as $ \omega = \omega(t) $,
is endowed with the Borel $ \sigma $-algebra $ \mathscr C $ generated by the open sets induced 
by the supremum metric. 

For a given probability space $ (\Omega, \mathscr F, P) $
we define the random function 
\begin{equation}\label{eq:rand-fun}
Y: (\Omega, \mathscr F) \mapsto (C[0,T], \mathscr C) \,\, .
 \end{equation}
We take a probability measure $ \nu $ on $ (C[0,T], \mathscr C) $ defined as
\begin{equation}
\nu(A) = P(Y^{-1} ( A)) \qquad A \in \mathscr C \,\, .
\end{equation}
%
%

For a set $ \Lambda \in \mathscr C $ such that $ \nu(\Lambda) > 0 $  we consider the space 
$
( \Lambda,  \mathscr C, \nu(\,\cdot \, | \Lambda ))
$
%
which is the trace of $ (C[0,T], \mathscr C, \nu) $ on the set $ \Lambda $,  where 
the conditional  probability measure $  \nu(\,\cdot \, | \Lambda ) : \Lambda \cap \mathscr C \mapsto [0,1]  $
is defined in the usual sense as
\begin{equation}
\nu(A  | \Lambda )= \frac{\nu(A \cap \Lambda )}{\nu(\Lambda)} \qquad A \in  \mathscr C \,\, .
\end{equation}
We then construct the space 
$
\big ( Y^{-1} (\Lambda),\mathscr F \cap  Y^{-1} (\Lambda), P ( \,\cdot \, | Y^{-1}(\Lambda) ) \big )
$
where
\begin{equation}
	P ( A | Y^{-1}(\Lambda) ) = \frac{P(A \cap Y^{-1} (\Lambda ) )}{P(Y^{-1} (\Lambda ) )} \qquad 
	\text{ for } A 
	\in \mathscr  F \cap  Y^{-1} (\Lambda) \,\, .
\end{equation}
\begin{definition}
	Given a random function $ Y  $ as in \eqref{eq:rand-fun} and a set $ \Lambda \in \mathscr C $ the \emph{conditional process} $ Y| \Lambda $ is defined  
	as the restriction of $ Y $ to the set $ \Lambda $:
	\begin{equation}
	Y|\Lambda : 
	\big (
	\,  Y^{-1} (\Lambda), \mathscr F  \cap  Y^{-1} (\Lambda), P (\, \cdot \,  | {Y^{-1}(\Lambda)} )  \, \big )
	\mapsto
	( \Lambda, \mathscr C, \nu(\, \cdot \,| \Lambda )  ) \,\, 
	\end{equation} 
\end{definition}
%
%

%
%

The following lemma provides the conditions for a conditional process to be Markov (see \cite{durrett77}). 

\begin{lemma}\label{lem:markov-cond}
	Let $ Y $ be a Markov process on $ C[0,T] $ and let $ \Lambda \in \mathscr C $ such that $ \nu(\Lambda) > 0 $. 
	Let $ \pi_{[0,t]} $ and  $ \pi_{[t,T]} $ be the projection maps on $ C[0,T] $ onto  $ C[0,t] $ and 
	 $ C[t,T] $, respectively. If for all $ t \in [0,T] $ there exist sets $ A_t \in \mathscr B(C[0,t] )$ and $ B_t \in \mathscr B(C[t,T] ) $ such that 
	 $ \Lambda   =   \pi_{[0,t]} ^{-1} A_t  \cap \pi_{[t,T]} ^{-1} B_t$ then $ Y|\Lambda  $ is Markov, 
	 where $ \mathscr B $ denotes the Borel sigma-algebra. 
\end{lemma}

In the following $ \nu(\,\cdot \,) $ denotes  the Wiener measure on $ (C[0,T], \mathscr C) $. 
For a Brownian motion starting at $ u $ we usually write $ P(\,\cdot \, | B(0) = u ) $  to denote $ P_u(\,\cdot \,) $
to underline the dependence on the starting point. 
The \emph{ drifted Brownian meander} can be represented as a conditional process $ B^\mu | \Lambda_{u,v} $ where 
 the conditioning event $ \Lambda_{u,v} $ is of the form
\[ \Lambda_{u,v} = \Big\{  \min_{ 0\leq z \leq t} B^{\mu}(z)  > v, B^\mu(0) = u\Big\} \, .\]
Analogously the \emph{Brownian excursion} is a conditional process $ B^\mu | \Lambda_{u,v,c}  $ with
\[ \Lambda_{u,v,c} = \Big\{ \min_{ 0\leq z \leq t} B^{\mu}(z)>v,\, B^\mu(0) = u, B^\mu(t) = c \Big\} 
\qquad u,c > v \,\, .
\]

We remark that the conditional processes introduced above are Markovian in light of \autoref{lem:markov-cond}.

For some fixed $ v > 0 $, we need to study the weak convergence of the measures 
$ \nu_{u,v} := \nu( \, \cdot \, | \Lambda_{u,v}) $ as $ u \downarrow v $. 
See \citet{billingsley2009convergence} for a treatise of the general theory of weak convergence. We here recall the main 
concepts we will make use of.

\begin{definition}
	Given a metric space $ (S, \rho) $ and a family $ \Pi $ of probability measures on $ (S, \mathscr B ( S)) $, 
	 $ \mathscr B ( S) $ being the Borel $ \sigma- $field on $ S $, we say that $ \Pi $ is tight if
	 \[
	 \forall \eta > 0 \quad  \exists \text{ compact } K \subset S  \quad  \text{s.t.} \quad \forall \nu \in \Pi \quad  \nu(K) > 1 - \eta 
	 .
	 \]
\end{definition}

The following theorem holds (see \cite{billingsley2009convergence}, Theorem 7.1, or \cite{karatzas2014brownian}, Theorem 4.15).
\begin{theorem}\label{thm:weak-conv}
	Let $ \{ X^{(n)} \}_n $ and $ X $ be stochastic processes on some probability space $ (\Omega, \mathscr F, P) $ onto $(C([0,T]), \mathscr C )$ and let $ \{\nu_n \}_n $ and $ \nu $, respectively, the induced measures. 
	If for every $ m $ and for every $ 0 \leq t_1 < t_2 < \cdots t_m \leq t $, 
	the finite dimensional distributions of $ (X^{(n)}_{t_1}, \ldots, X^{(n)}_{t_m} ) $ converge 
	to those of $ (X_{t_1}, \ldots, X_{t_m} ) $ and the family $ \{\nu_n\}  $ is tight then $ \nu_n \Rightarrow \nu $.
\end{theorem}
%
%
%
%
%
%
%

%
%
%
%

%
%

%
Exploiting these tools one is able to assess the existence of some process whose finite dimensional distributions 
coincide with those of the weak limit $ \nu_v $ of $ \nu_{u,v}$ when $ u \downarrow v $. This measure will coincide with that induced by
$ B^\mu | \Lambda_v $ where 
\[
\Lambda_v = \Big \{\omega : \inf_{ 0 < z < T } \omega(z) > v , \omega(0)  = v \Big \}.
\]

The following result holds concerning the weak convergence to the Brownian meander with 
drift(\cite{iafrate19})

\begin{theorem}\label{thm:weak-conv-mdr}
	The following weak limit holds:
	\begin{equation}\label{eq:weak-conv-mdr}
	B^\mu \Big |\Big \{  \min_{ 0 \leq z \leq t } B^\mu > v, B^\mu(0) = u \Big \} \xRightarrow[u\downarrow v]{}
	B^\mu \Big | \Big \{ \inf_{ 0 < z < t } B^\mu > v, B^\mu(0) = v \Big \}
	\end{equation}
\end{theorem}

This means that the continuous mapping theorem holds, i.e.
$ \nu_{u,v} \circ g^{-1}$  \raisebox{-3pt} {$\xRightarrow{u\to v }$ }$  \nu_{v} \circ g^{-1} $ for any bounded uniformly continuous function $ g $.

\section{Sojourn time of a generalized  Brownian meander }

In this section we study the random time spent by the drifted Brownian motion on the half-line $[0, \infty)$ up to time $t$ under the condition 
that in $ [0,l] $, $l<t$ the conditions 
\begin{equation}\label{eq:souj-conds}
\left\{ \min_{0\leq z \leq l} B^\mu(z)>0 \right\} \,\, \mathrm{ and }
\,\,
\left\{ B^\mu(0) = u \right\} 
\end{equation}
are fullfilled. 

In other words in $ [0,l] $ the particle behaves like a Brownian meander
and then is free to move on the whole line.

We therefore study the distribution of
\begin{equation}\label{eq:souj-def}
\Gamma^\mu_{(0,t)}  = l + \Gamma^\mu_{l,t} = l + \int_{l}^{t} \mathbbm 1_{[0, \infty)}( B^\mu(s)) \mathrm d s 
\end{equation}
The conditions \eqref{eq:souj-conds}
exert their effect on the distribution of the sojourn time $ \Gamma_{l,t}$
because the position at time $ l $ of the Brownian particle is random.

It is well-known that for the non-drifted Brownian bridge the following 
result hold
\begin{equation}\label{eq:souj-old}
P\left\{  \Gamma_t \in \mathrm d s \Big | B(0)  =B(t) = 0    \right\} = \frac{ \mathrm d s }{t}  
\qquad 0 < s < t \,\,.
\end{equation}

If the starting point is $\left\{ B(0) = u \right\}$ than 
the distribution of the sojourn time becomes
\begin{align}\label{eq:souj-new}
&P\left\{  \Gamma_t \in \mathrm d s \Big | B(0)=u, B(t) = 0    \right\} =
\\
&=\sqrt \frac{t}{2 \pi} e^{  \frac{ u^2}{2t}} \int_{0}^{s} \frac{ u e ^ {  - \frac{  u^2}{2w}}}{ \sqrt{ w^3 (t-w)^3}} \,\mathrm d w
\qquad 0\leq s \leq t \,\,,\,\,\, u > 0. \notag 
\end{align}

For $ u \downarrow 0 $ the first-passage time distribution of a Brownian
motion 
\[
f(u,w) = \frac{u e^{- \frac{u^2}{2w}}}{\sqrt{2 \pi w^3}}
\rightarrow \delta(w)
\]
and from \eqref{eq:souj-new} result \eqref{eq:souj-old} is 
retrieved as a special case. 

For the sojourn time of the drifted Brownian motion it holds that (see (3.1) of \cite{BEGHIN2003291})
\begin{align}\label{eq:souj-joint}
	P\left\{ \Gamma^\mu \in \mathrm d s,  B^\mu(t) \in \mathrm d x \right\} = 
	\left\{ 
	\begin{aligned}
	&
	\frac{\mathrm d s \mathrm d x}{2 \pi } e^{ - \frac{\mu^2 t}{2} + \mu x}
	\int_{t-s}^t \frac{x}{\sqrt{z^3 (t-z)^3}} e^{ - \frac{x^2}{2(t-z)}} \, 
	\mathrm d z & x > 0 
	\\
	&
	\frac{\mathrm d s \mathrm d x}{2 \pi } e^{ - \frac{\mu^2 t}{2} + \mu x}
	\int_{s}^t \frac{-x}{\sqrt{z^3 (t-z)^3}} e^{ - \frac{x^2}{2(t-z)}} \, 
	\mathrm d z & x < 0
	\end{aligned}
	\right.
\end{align}

In the book by Borodin and Salminen (1996) there is an 
alternative form to \eqref{eq:souj-joint} (formula 2.1.4.8, page 203)
which reads
\begin{align}\label{eq:souj-joint-alt}
P\left\{ \Gamma^\mu \in \mathrm d s,  B^\mu(t) \in \mathrm d x \right\} = 
\left\{ 
\begin{aligned}
&
\frac{\mathrm d s \mathrm d x}{2 \pi } e^{ - \frac{\mu^2 t}{2} + \mu x}
\int_{0}^\infty \frac{z(z + x)}{\sqrt{s^3 (t-s)^3}} 
e^{ - \frac{(z +x)^2}{2s} -   \frac{z^2}{2(t-s)} } \, 
\mathrm d z & x > 0 
\\
&
\frac{\mathrm d s \mathrm d x}{2 \pi } e^{ - \frac{\mu^2 t}{2} + \mu x}
\int_{0}^\infty \frac{z(z - x)}{\sqrt{s^3 (t-s)^3}} 
e^{  \frac{z^2}{2s } - \frac{(z -x)^2}{2(t-s)}   } \,
\mathrm d z
& x<0
 \end{aligned}
\right.
\end{align}

Result \eqref{eq:souj-joint} has been obtained in Beghin, Nikitin, Orsingher  (2003) by applying the conditional Feynman--Kac functional. 
It can also be obtained by applying the Girsanov theorem to the non-drifted
joint distribution of $ (\Gamma_t, B(t)) $.

We here consider the  family of measures $\{ \nu_{u}^l , u>0  \}$ defined as
\begin{equation}\label{eq:souj-cond-meas-def}
\nu_{u}^l(A) = P\Big\{ B^\mu \in A \Big | \Lambda_{u}^l \Big \} \qquad A \in \mathscr C, \quad u > 0
\end{equation}
where
\[
\Lambda_{u}^l = \{  \omega \in C[0,t]:  \min_{ 0\leq z \leq l} \omega(z) > v, \omega(0) = u \}
\]
The tightness of the family of measures \eqref{eq:souj-cond-meas-def}
can be proved following the same procedure as in \cite{iafrate19}.
In this way we can state the weak convergence of the measures
\[
\nu_{u}^l \Rightarrow_{u \downarrow v} \nu_{v}^l
\]
holds 
where 
\[
\nu_{v}^l(A) = P\Big(B^\mu \in A \Big | \inf_{ 0<  z \leq l} B^\mu(z) > v, B^\mu(0) = v \Big ) \qquad A \in \mathscr C
\]

For our purposes we assume that $ v = 0 $.

Our main concern here is  the study of the following  
sequence of probability measures

\begin{equation}\label{eq:souj-dist-prethm}
	P\left\{  \Gamma^\mu_{l,t} \in \mathrm d s \Big |  \min_{0\leq z\leq l} B^\mu(z) > 0, B^\mu(0)=u  \right\}
	\end{equation}

\begin{equation}\label{eq:souj-exc-dist-pre}
P\left\{  \Gamma^\mu_{l,t} \in \mathrm d s \Big |  \min_{0\leq z\leq l} B^\mu(z) > 0, B^\mu(0)=u, B^\mu(t) = 0	 \right\}
\end{equation}

$ \quad 0 < s < t-l, u>0,  l < t $. 
In the next theorem we give the distribution $ \eqref{eq:souj-dist-prethm} $, 
for $ u > 0 $. The distribution \eqref{eq:souj-exc-dist-pre} is treated in Section 4.

\begin{theorem}
	\begin{align}\label{eq:souj-thm}
	&P\left\{  \Gamma^\mu_{l,t} \in \mathrm d s \Big |  \min_{0\leq z\leq l} B^\mu(z) > 0, B^\mu(0)=u  \right\}	\\
	&=
	\frac{
		\int_0^\infty 
		P\left\{ \min_{0\leq z\leq l} B^\mu(z) > 0, 
		B^\mu(l) \in \mathrm d y \Big | B^\mu(0) = u \right\}
		P\left\{ \Gamma^\mu_{l,t} \in \mathrm d s \Big | 
		 B^\mu(l) = y \right\}
	}{
		P\left\{ \min_{0\leq z\leq l} B^\mu(z) > 0, 
		 \Big | B^\mu(0) = u \right\}
	}\notag 
	\end{align}
	
	where 
	\begin{align*}
	&\frac{ 
		P\left\{ \min_{0\leq z\leq l} B^\mu(z) > 0, 
		B^\mu(l) \in \mathrm d y \Big | B^\mu(0) = u \right\}
	}{
	P\left\{ \min_{0\leq z\leq l} B^\mu(z) > 0, 
	\Big | B^\mu(0) = u \right\}
	}
	\numberthis 
	\\
	&=
	\frac{
		\frac{1}{\sqrt{2 \pi l}}\left(
		e^{-\frac{(y-u)^2}{2l}} - e^{-\frac{(y+u)^2}{2l}} \right)
		e^{ - \frac{\mu^2 l}{2} + \mu (y-u) } \, \mathrm d y
	}{
		\int_0^\infty 
		\frac{1}{\sqrt{2 \pi l}}\left(
		e^{-\frac{(y-u)^2}{2l}} - e^{-\frac{(y+u)^2}{2l}} \right)
		e^{ - \frac{\mu^2 l}{2} + \mu (y-u) }
		\,\mathrm dy
	}
	\end{align*}
	
	and 
	
	\begin{align*}\label{eq:dist-souj-drift-ac-thm}
	&
	\numberthis
	P\left\{ \Gamma_{l,t}^\mu \in \mathrm d s \Big | B^\mu(l) = y \right\}
	\\
	&=
	e^{- \frac{\mu^2 (t-l)}{2} + \mu y}\cdot
	2 e^{\mu y}
	\int_y^\infty e^{-\mu w} \mathrm d w
	\int_0^\infty e^{\mu w'} 
	\frac{w e^{- \frac{w^2}{2s}}}{\sqrt{2 \pi s^3}} 
	\frac{w' e^{- \frac{w'^2}{2(t-l -s)}}}{\sqrt{2 \pi (t-l -s)^3}} \,
	\mathrm d w'
	\mathrm d s .
	\end{align*}
\end{theorem}

\begin{proof}

The evaluation of the probability $ P\left\{ \Gamma^\mu_{l,t} \in \mathrm d s \Big | 
B^\mu(l) = x \right\} $ can be performed applying  the
Feynman--Kac functional. In order to simplify the notation we consider 
$ \Gamma_{0,t}^\mu = \Gamma_t^\mu $ and the functional 
\begin{equation}\label{eq:fk-functional}
	w(x,t) = \mathbb E \left\{
	e^{ - \int_0^t k(B^\mu (s)) \mathrm d s  } \Big | 
	B^\mu(0) = x
	\right\}
\end{equation}
which solves the Cauchy problem 
\begin{equation}\label{eq:fk-equation}
	\left\{ 
	\begin{aligned}
	&\frac{\partial w}{\partial t}  = 
	\frac 12 \frac{\partial ^2 w}{\partial x^2} -
	\mu \frac{\partial w}{\partial x} - k(x) w \\
	& w(x,0) = 1
	\end{aligned}
	\right.
\end{equation}
for 
\[
k(x) = \begin{cases}
\beta & x > 0 \\
0 & x<0
\end{cases}
\]

The functional \eqref{eq:fk-functional} coincides with the Laplace
transform of the sojourn time

\begin{equation}
 \Gamma^\mu_{t} =  \int_{0}^{t} \mathbbm 1_{[0, \infty)}( B^\mu(s)) \mathrm d s 
\end{equation}

By means of the transformation
\begin{equation}
	w(x,t) = e^{ - \frac{\mu^2t}{2} + \mu x} z(x,t)
\end{equation}
the problem \eqref{eq:fk-equation} is converted into 

\begin{equation}\label{eq:fk-equation-tr}
\left\{ 
\begin{aligned}
&\frac{\partial z}{\partial t}  = 
\frac 12 \frac{\partial ^2 z}{\partial x^2} -
\beta \mathds 1_{[0, \infty)}(x) z \\
& z(x,0) = e^{-\mu x} \,\,
\end{aligned}
\right.
\end{equation}

For the Laplace transform $ Z(x, \gamma) = \int_0^\infty e^{ - \gamma t} z(x, t) \,\mathrm d t $ we have

\begin{equation}
	\gamma Z - z(x, 0) = 
	\frac 12 \frac{\mathrm d^2 Z}{\mathrm d x^2} 
	- \beta \mathds 1_{[0, \infty)}(x) Z \,\, .
\end{equation}

By taking into account that we need a bounded solution of
\eqref{eq:fk-equation} and that $ z(x,0) = e^{-\mu x} $
we get
\begin{equation}\label{eq:fk-lap-without-const}
	Z(x, \gamma) = 
	\left\{
	\begin{aligned}
	&
	B e^{ - x \sqrt{2(\gamma + \beta)}  } + 
	\frac{e^{- \mu x}}{\beta + \gamma - \frac{\mu^2}{2}}
	& x > 0
	\\
	&
	C e^{ x\sqrt{2 \gamma}}  + 
	\frac{e^{\mu x}}{\gamma - \frac{\mu^2}{2}}
	&
	x < 0\,\, 
	\end{aligned}
	\right.
\end{equation}

The continuity of $ Z $ and $ \frac{\mathrm d Z}{\mathrm d x } $ at $ x = 0 $
imply that 
\begin{equation}
B = \frac{2}{\sqrt{2(\beta + \gamma)} + \mu}
\left( \frac{1}{\sqrt{2 \gamma} - \mu} - \frac{1}{\sqrt{2 (\beta +\gamma)} - \mu} \right)
\end{equation}

From \eqref{eq:fk-lap-without-const}, for $ x > 0 $, we have therefore
\begin{align}\label{eq:fk-a}
	Z(x, \gamma)   &= 
	\frac{2e^{ - x \sqrt{2(\gamma + \beta)}  }}{\sqrt{2(\beta + \gamma)} + \mu}
	\left( \frac{1}{\sqrt{2 \gamma} - \mu} - \frac{1}{\sqrt{2 (\beta +\gamma)} - \mu} \right) 
	 + 
	\frac{e^{- \mu x}}{\beta + \gamma - \frac{\mu^2}{2}}
	\\
	&=
	2 e^{\mu x}\int_x^\infty e^{-w ( \sqrt{2(\gamma + \beta)} + \mu)}
	\left( \frac{1}{\sqrt{2 \gamma} - \mu} - \frac{1}{\sqrt{2 (\beta +\gamma)} - \mu} \right) \mathrm d w + 
	\frac{e^{- \mu x}}{\beta + \gamma - \frac{\mu^2}{2}}
	\notag 
\end{align}
In order to develop \eqref{eq:fk-a} we need the following expressions

\begin{gather*}
	e^{-w ( \sqrt{2(\gamma + \beta)} + \mu)}
	= 
	e^{ - \mu w} \int_0^\infty e^{- (\gamma + \beta ) t}
	\frac{w e^{- \frac{w^2}{2t}}}{\sqrt{2 \pi t^3}} \,\mathrm d t
	\\
	\frac{1}{\sqrt{2 \gamma } - \mu} = 
	\int_0^\infty e^{-w(\sqrt{2\gamma} -\mu )} \mathrm d w = 
	\int_0^\infty e^{\mu w} \mathrm dw  \int_0^\infty e^{- \gamma t}
	\frac{w e^{- \frac{w^2}{2t}}}{\sqrt{2 \pi t^3}} \,\mathrm d t
	\\
	\frac{1}{\sqrt{2 (\gamma + \beta) } - \mu} = 
	\int_0^\infty e^{-w(\sqrt{2\gamma} -\mu )} \mathrm d w = 
	\int_0^\infty e^{\mu w} \mathrm d w \int_0^\infty e^{- (\gamma + \beta) t}
	\frac{w e^{- \frac{w^2}{2t}}}{\sqrt{2 \pi t^3}} \,\mathrm d t
\end{gather*}

By composing all these terms we obtain $ Z(\gamma, x) $
explicitly as
\begin{align*}\label{eq:fk-b}
	Z(x, \gamma) &= 2 e^{\mu x}\Bigg[
	\int_x^\infty e^{-\mu w} \mathrm d w
	 \int_0^\infty e^{- (\gamma + \beta) t}
	\frac{w e^{- \frac{w^2}{2t}}}{\sqrt{2 \pi t^3}} \,\mathrm d t
	\int_0^\infty e^{\mu w'}  \mathrm d w'
	 \int_0^\infty e^{- \gamma t'}
	\frac{w' e^{- \frac{w'^2}{2t'}}}{\sqrt{2 \pi t'^3}} \,\mathrm d t'
	\\
	& \qquad -
	\int_x^\infty e^{-\mu w} \mathrm d w
	 \int_0^\infty e^{- (\gamma + \beta) t}
	\frac{w e^{- \frac{w^2}{2t}}}{\sqrt{2 \pi t^3}} \,\mathrm d t
	\int_0^\infty e^{\mu w'}  \mathrm d w'
	\int_0^\infty e^{- (\gamma +\beta) t'}
	\frac{w' e^{- \frac{w'^2}{2t'}}}{\sqrt{2 \pi t'^3}} \,\mathrm d t'
	\Bigg]\\
	& \qquad
	 + e^{-\mu x} \int_0^\infty e^{-w(\beta + \gamma - \frac{\mu^2}{2})}
	\,\mathrm d w .
	\numberthis
\end{align*}

The inverse Laplace transform of \eqref{eq:fk-b} after
a quick check becomes

\begin{align*}\label{eq:fk-c}
\numberthis 
	z(x, t) &= 2 e^{\mu x}
		\int_x^\infty e^{-\mu w} \mathrm d w
		\int_0^\infty e^{\mu w'}  \mathrm d w'
		 \int_0^t
	\frac{w e^{- \frac{w^2}{2s}}}{\sqrt{2 \pi s^3}} 
	\frac{w' e^{- \frac{w'^2}{2(t-s)}}}{\sqrt{2 \pi (t-s)^3}} \,
	(e^{-\beta s} - e^{-\beta t})
	\mathrm d s
	\\
	& \qquad +
	e^{-\mu x} \int_0^\infty e^{- w \beta} e^{\mu^2 \frac w2}
	\delta(t-w)
	\,\mathrm d w
	\\
	&=
	2 e^{\mu x}
		\int_x^\infty e^{-\mu w} \mathrm d w
	\int_0^\infty e^{\mu w'}  \mathrm d w'
	\int_0^t \beta e^{-\beta z} \,\mathrm d z
	\int_0^z
	\frac{w e^{- \frac{w^2}{2s}}}{\sqrt{2 \pi s^3}} 
	\frac{w' e^{- \frac{w'^2}{2(t-s)}}}{\sqrt{2 \pi (t-s)^3}} \,
	\mathrm d s
	\\
	& \qquad +
	e^{-\mu x} e^{- t \beta } e^{\frac{\mu^2t}{2}}
	\qquad \text{(by integrating by parts with respect to } z)
	\\
	&=
	2 e^{\mu x}
	\int_x^\infty e^{-\mu w} \mathrm d w
	\int_0^\infty e^{\mu w'}  \mathrm d w'
	\int_0^t  e^{-\beta z} 
	\frac{w e^{- \frac{w^2}{2z}}}{\sqrt{2 \pi z^3}} 
	\frac{w' e^{- \frac{w'^2}{2(t-z)}}}{\sqrt{2 \pi (t-z)^3}} 
	\,\mathrm d z
	\\
	& \qquad -
	2 e^{\mu x} e^{- \beta t}
	\int_x^\infty e^{-\mu w} \mathrm d w
	\int_0^\infty e^{\mu w'}  \mathrm d w'
	\int_0^t 
	\frac{w e^{- \frac{w^2}{2s}}}{\sqrt{2 \pi s^3}} 
	\frac{w' e^{- \frac{w'^2}{2(t-s)}}}{\sqrt{2 \pi (t-s)^3}} \,
	\mathrm d s
	\\
	& \qquad +
	e^{-\mu x} e^{- t \beta } e^{\frac{\mu^2t}{2}}.
\end{align*}
In order to obtain the absolutely continuous part of the
distribution of $ \Gamma_t^\mu  $ we need to multiply 
the inverse Laplace transform of \eqref{eq:fk-c} 
(with respect to $ \beta $)
by the Girsanov term $ e^{- \frac{\mu^2 t}{2} + \mu x} $. 

In conclusion for $ 0 \leq s < t  $ we have that
\begin{align*}\label{eq:dist-souj-drift-ac}
	&
	\numberthis
	P\left\{ \Gamma_t^\mu \in \mathrm d s \Big | B^\mu(0) = x \right\}
	\\
	&=
	e^{- \frac{\mu^2 t}{2} + \mu x}\cdot
	2 e^{\mu x}
	\int_x^\infty e^{-\mu w} \mathrm d w
	\int_0^\infty e^{\mu w'} 
	\frac{w e^{- \frac{w^2}{2s}}}{\sqrt{2 \pi s^3}} 
	\frac{w' e^{- \frac{w'^2}{2(t-s)}}}{\sqrt{2 \pi (t-s)^3}} \,
	 \mathrm d w'
	\mathrm d s .
\end{align*}

The singular part of the distribution is instead
\begin{align*}
	P\left\{ \Gamma_t^\mu = t \Big | B^\mu(0) = x \right\} = 
	1 - 2e^{- \frac{\mu^2t}{2} + \mu x} 
	\int_x^\infty e^{- \mu w}\,\mathrm d w
	\int_0^\infty e^{\mu w' } \, \mathrm d w'
	\int_0^t
		\frac{w e^{- \frac{w^2}{2s}}}{\sqrt{2 \pi s^3}} 
	\frac{w' e^{- \frac{w'^2}{2(t-s)}}}{\sqrt{2 \pi (t-s)^3}} \,
	\mathrm d s
\end{align*}
For $ \mu = 0 $ we get
\begin{equation}\label{eq:souj-dist-start-ac}
	P\left\{ \Gamma_t \in \mathrm d s \Big | B(0) = x \right\} = 
	\frac{e^{- \frac{x^2}{2s}}}{\pi \sqrt{s(t-s)} } \mathrm d s
	\qquad 0 \leq s <t
\end{equation}
\begin{equation}\label{eq:souj-dist-start-sing}
P\left\{ \Gamma_t = t \Big | B(0) = x \right\} = 
1 -  \int_0^t 
\frac{e^{- \frac{x^2}{2s}}}{\pi \sqrt{s(t-s)} }\mathrm d s
= \frac{2}{\sqrt{2 \pi t}} \int_0^x e^{ - \frac{w^2}{2t}} \,\mathrm d w .
\end{equation}

\end{proof}

\begin{remark}
	The distribution \eqref{eq:dist-souj-drift-ac} can be written as 
	\begin{equation}
	P\left\{ \Gamma_t^\mu \in \mathrm d s \Big | B^\mu(0) = x \right\}
	= 
	\frac 12 e^{- \frac{\mu^2t}{2} + \mu x}
	P_0\left\{ B_1^{el} (s) \in \mathrm d x \right\}
	P_0\left\{ B_1^{el} (t-s) \in \mathrm d 0 \right\}
	\end{equation}
	where $  B_1^{el}  $ and $  B_2^{el}  $
	are independent elastic Brownian motions. 
	The elastic Brownian motion $ B_1^{el} $
	runs on the positive half-line and is defined as 
	\[ 
	B_1^{el} = 
	\begin{cases}
	B^+(t) \qquad & t < T \\
	0 & t\geq T
	\end{cases} \]
	where $ B^+ $ is a reflecting Brownian motion
	and $ T $ is a random time such that
	\begin{equation}\label{eq:elastic-surv-dist}
	P(T>t | \mathscr F_t) = e^{- \mu L(0,t)} \qquad \mu > 0
	\end{equation}
	and $ L(0,t) $ is the local time at 0. The killing rate appearing in
	\eqref{eq:elastic-surv-dist} coincides with the drift $ \mu $. 
	Since
	\begin{align*}
		2 \int_0^\infty e^{\mu w} 
		\frac{w e^{- \frac{w^2}{2(t-s)}}}{\sqrt{2 \pi (t-s)^3}}
		\mathrm d w = 2\int_{-\infty}^0e^{ - \mu w} 
		\frac{|w| e^{- \frac{w^2}{2(t-s)}}}{\sqrt{2 \pi (t-s)^3}}
		\mathrm d w
	\end{align*}
	the second term inside \eqref{eq:dist-souj-drift-ac} can be interpreted 
	as the probability that an elastic Brownian motion $ B^{el}_2 $
	running on the negative half-line, starting at zero and visiting
	the barrier at time $ t-s $. 
	For $ \mu < 0 $ the elastic Brownian motions must be interchanged.
	
	For $ B^{el}(t), t> 0 $ starting from an arbitrary 
	$ y>0 $ the transition density develops as
	\begin{align*}
	p^{el}(x, t;y,0) 
	&= 
	\frac{
	e^{- \frac{(x-y)^2}{2t} } - 
	e^{- \frac{(x+y)^2}{2t} }
	}{\sqrt{2 \pi t}} + 
	2 e^{\mu (x+y)} \int_{x+y}^\infty e^{-\mu w} 
	\frac{w e^{- \frac{w^2}{2t}}}{\sqrt{2 \pi t^3}} 
	\,\mathrm d w
	\\
	&=
	\frac{
		e^{- \frac{(x-y)^2}{2t} } + 
		e^{- \frac{(x+y)^2}{2t} }
	}{\sqrt{2 \pi t}} -
	2\mu  e^{\mu (x+y)} \int_{x+y}^\infty e^{-\mu w} 
	\frac{ e^{- \frac{w^2}{2t}}}{\sqrt{2 \pi t}} 
	\,\mathrm d w
	\end{align*}
	and thus for $ y = 0 $ we obtain the distribution of 
	$ B_1^{el}(t) $.

	We extract from Theorem 3.1 the distribution of 
	$ \Gamma_{l,t} $ for $ \mu = 0 $ and $ u = 0 $ in the next theorem. 
\end{remark}

\begin{theorem}
	
	\begin{align}\label{eq:souj-mdr-thm-ac}
		P\left\{ \Gamma_{l,t} \in \mathrm d s |
		\inf_{ 0< z \leq l} B (z) > 0,  B(0) = 0  \right\}
		= 
		\frac{\mathrm d s}{\pi \sqrt{s(t-l-s)}} \frac{s}{s+l}
		\qquad 0 < s < t-l
	\end{align}
	
	and 
	\begin{align}\label{eq:souj-mdr-thm-sing}
	P\left\{ \Gamma_{l,t} =  t-l |
	\inf_{ 0\leq z \leq l} B (z) > 0,  B(0) = 0  \right\}
	= 
	\sqrt \frac lt. 
	\end{align}
	
\end{theorem}
\begin{proof}
	By taking the limit for $ u \downarrow 0 $ and assuming
	$ \mu = 0 $ in \eqref{eq:souj-thm} we get
		
		\begin{align}\label{eq:souj-proof-b}
		&P\left\{  \Gamma_{l,t} \in \mathrm d s \Big |  \inf_{0\leq z\leq l} B(z) > 0, B(0)=0  \right\} 
		\\
		&=
		\lim_{u\downarrow 0}
		\frac{ \displaystyle
			\int_{0}^{\infty}
			\frac{e^{- \frac{ (y - u )^2}{2 l } } 
				- e^{  - \frac{ (y + u)^2}{2 l }   } }
			{ \sqrt{2\pi l } }  
			P\left\{ 
			\Gamma_{l,t}	\in \mathrm ds \Big| B(l) = y \right\}
			\mathrm d y
		} {\displaystyle
		\int_0^\infty \frac{e^{- \frac{ (y - u )^2}{2 l } } 
			- e^{  - \frac{ (y + u)^2}{2 l }   } }
		{ \sqrt{2\pi l } }  \mathrm d y
		} \, \mathrm d s \notag 
		\\
		&=
		\mathrm ds
		\int_{0}^{\infty}
		\frac yl e^{ - \frac{y^2}{2l}}
		\frac{ e ^ { - \frac{ y^2}{2s}}}{\pi \sqrt{s(t-l-s)}  } 
		\, \mathrm d w \notag \\
		&=
		\frac{\mathrm d s }{\pi  \sqrt {s( t- l - s)}} \,  \frac{s }{ s + l} \qquad 0 \leq s <t-l
		\notag 		
		\end{align}
		For $ s = t-l $, $ \mu = 0 $ and letting 
		$ u \downarrow 0 $ we have instead
		\begin{align}\label{eq:souj-proof-c}
		&
		\lim_{u \downarrow 0}
		P\left\{  \Gamma_{l,t}  = t-l \Big |  \min_{0\leq z\leq l} B(z) > 0, B(0)=u  \right\}
		\\
		&=
		\int_{0}^{\infty}
		\frac{ w}{l} e^{ - \frac{w^2}{2l}}
		\int_{0}^{w}
		\frac{2 }{ \sqrt{2 \pi (t-l) }}
		e ^ { - \frac{ z^2}{2(t-l)}} \mathrm d z
		\mathrm d w
		 \notag \\
		&=
		\frac{2 }{ \sqrt{2 \pi (t-l) }}
		\int_{0}^{\infty}
		e ^ { - \frac{ z^2}{2(t-l)}} 
		\int_{z}^{\infty}
		\frac wl e^{ - \frac{w^2}{2l}}  \,\mathrm d w \,\mathrm d z \notag \\
		&=
		\frac{2 }{ \sqrt{2 \pi (t-l) }}
		\int_{0}^{\infty}
		e ^ { - \frac{ z^2}{2(t-l)} - \frac{ z^2}{2 l}} 
		\,\mathrm d z \notag \\	 
		&= 
		\sqrt  \frac lt \notag 
		\end{align}
\end{proof}

\begin{remark}
	The distribution \eqref{eq:souj-mdr-thm-ac} can also be obtained as 
	\begin{align*}
		&\int_0^\infty P\left\{  B(l) \in \mathrm d y \Big | \inf_{ 0< z \leq l} B(s) > 0, B(0) = 0
		 \right\}
		P\left\{   \Gamma_{l,t} \in \mathrm d s \Big | B(l) = y  \right\}   
		\\
		&=
		\mathrm d s \int_0^\infty \frac yl e^{ - \frac{y^2}{2l}} 
		\frac{e^{ - \frac{y^2}{2s}} }{\pi \sqrt{s (t-l-s)}}
		\mathrm d y
		=  
		\frac{\mathrm ds }{\pi  \sqrt{s (t-l-s)}} \frac{s}{s+l}
	\end{align*}
	This is because at time $ l $ the meander has a Rayleigh distributed position and
	then formula \eqref{eq:souj-dist-start-ac} is applied.
\end{remark}

\begin{remark}
	We can write the distribution function
	of $ \Gamma_{l,t}  $ as follows
	\begin{align*}
	P\left\{ \Gamma_{l,t} < z
	 \Big |  \inf_{0\leq z\leq l} B(z) > 0, B(0)=0 
	 \right\} &=
	\frac 1\pi \int_{0}^{z} \sqrt \frac{s}{t-l-s}\, \frac{\mathrm d s }{l+s} 
	\\
	&=
	\frac 1\pi \int_{0}^{\frac{z}{t-l}} \sqrt \frac{s}{1-s}\, \frac{ t-l }{l+ (t-l)s} \mathrm d s \\
	&=
	2\frac {t-l}\pi \int_{0}^{\arcsin \sqrt \frac{z}{t-l}} 
	\frac{  \sin^2 \varphi }{t \sin^2 \varphi + l \cos^2 \varphi } \mathrm d \varphi \\
	&=
	\frac {2}\pi (t-l) \int_{0}^{\arcsin \sqrt \frac{z}{t-l}} 
	\frac{  \tan^2 \varphi }{t \tan^2 \varphi + l  } \mathrm d \varphi \\
	&=
	\frac {2}\pi (t-l) \int_{0}^{ \sqrt \frac{z}{t-l-z}} 
	\frac{ y^2 }{t y^2 + l  } \cdot \frac{\mathrm d y}{1 + y^2}  \\  
	&=
	\frac {2}\pi  \int_{0}^{\sqrt \frac{z}{t-l-z}} 
	\left\{ 
	\frac{1}{1 + y^2} - \frac{ l }{t y^2 + l  }  \right\} \, \mathrm d y\\ 
	&=
	\frac 2 \pi \arctan \sqrt \frac{z}{t-l-z} 
	- \frac 2 \pi \sqrt \frac lt \arctan \left( \sqrt \frac tl \sqrt \frac{z}{t-l-z}\right)
	\\
	&=
	\frac 2 \pi \arcsin \sqrt \frac z {t-l} - \frac 2 \pi \sqrt \frac lt
	\arcsin \sqrt \frac{zt}{(l+z)(t-l)}
	\end{align*}
	for $  0 < z < t-l $.
	Clearly for $ z = t-l $ we have that
    \[
    P\left\{ \Gamma_{l,t}< t-l\right\} = 1 - \sqrt \frac lt.
    \]
\end{remark}
%
%
%
%

\section{Sojourn time of the generalized excursion process with drift}

We now consider the conditional process $ B^\mu| \Lambda_{u,0}^l$ where
\begin{equation}\label{eq:souj-exc-cond}
\Lambda_{u,0}^l = \{  \omega \in C[0,t]:  \min_{ 0\leq z \leq l} \omega(z) > 0, \omega(0) = u , \omega(t)=0\}
\end{equation}

We introduce the family of measures $\{ \nu_{u,0}^l , u>0  \}$ defined as
\[
\nu_{u,0}^l(A) = P\Big(B^\mu \in A \Big | \Lambda_{u,0}^l \Big ) \qquad A \in \mathscr C, \quad u > 0
\]

Arguing as in \autoref{thm:weak-conv-mdr} we have that the weak convergence
$
\nu_{u,0}^l \Rightarrow_{u \downarrow 0} \nu_{0,0}^l
$
holds 
where 
\[
\nu_{0,0}^l(A) = P\Big(B^\mu \in A \Big | \inf_{ 0<  z \leq l} B^\mu(z) > 0, B^\mu(0) = 0 , B^\mu(t)=0\Big ) \qquad A \in \mathscr C, \quad u > 0.
\]

We study below the sojourn time $\Gamma^\mu_{l,t} = \int_{l}^{t} \mathbbm 1_{[0, \infty)}( B^\mu(s)) \mathrm d s  $ under the condition $\Lambda_{u,0}^l$, i.e.  when the process is conditioned to
remain positive up to time $ l $ and to return to zero at  a subsequent time $ t $.
In particular, we give the explicit limiting distribution 
of the sojourn functional $\Gamma^\mu_{l,t}$ under the condition $ \Lambda_{0,0}^l $, that is  when $ u $ approaches the zero level.

It is well known that the excursion process with drift, 
that is the bridge of the Brownian meander, is not affected by 
the drift $ \mu $ (see for example \cite{iafrate19}). We here show that this is also true for the 
distribution of the sojourn time $\Gamma^\mu_{l,t}$. 

The explicit distribution of $\Gamma^\mu_{l,t} $
under the 
condition \eqref{eq:souj-exc-cond} is 
\begin{align*}\label{eq:souj-exc-dist}
	& \numberthis 
	P\left\{ \Gamma_{l,t}^\mu \in \mathrm d s
	\Big | \min_{ 0\leq z \leq l} B^\mu(z) > 0, B^\mu(0) = u, B^\mu(t)  = 0
	\right\} 
	\\
	&= \lim\limits_{w \downarrow 0}
	\frac{
	\int_0^\infty 
	P\left\{ 
	 \min_{ 0\leq z \leq l} B^\mu(z) > 0, B^\mu(l) \in \mathrm d y
	 \Big | B^\mu(0) = u\right\}
	 P\left\{ \Gamma_{l,t}^\mu \in \mathrm d s , B^\mu(t) \in \mathrm dw
	 \Big | B^\mu(l) = y \right\}
	}{
	\int_0^\infty 
	P\left\{ 
	\min_{ 0\leq z \leq l} B^\mu(z) > 0, B^\mu(l) \in \mathrm d y
	\Big | B^\mu(0) = u \right\}
	P\left\{ 
	B^\mu(t) \in \mathrm d w \Big | B^\mu(l) = y \right\}
	}
\end{align*}

Clearly

\begin{align*}\label{eq:conditional-min}
	P\left\{ \min_{0\leq z\leq l} B^\mu(z) > 0, 
	B^\mu(l) \in \mathrm d y \Big | B^\mu(0) = u \right\}
\numberthis 
&=
	\frac{1}{\sqrt{2 \pi l}}\left(
	e^{-\frac{(y-u)^2}{2l}} - e^{-\frac{(y+u)^2}{2l}} \right)
	e^{ - \frac{\mu^2 l}{2} + \mu (y-u) } \,\mathrm d y
\end{align*}

%
For our analysis we must consider that

\begin{align*}\label{eq:conditional-min-lim}
 \numberthis 
&\lim\limits_{u \downarrow 0}
\frac{
	P\left\{ 
	\min_{ 0\leq z \leq l} B^\mu(z) > 0, B^\mu(l) \in \mathrm d y
	\Big | B^\mu(0) = u\right\}
}{
	\int_0^\infty 
	P\left\{ 
	\min_{ 0\leq z \leq l} B^\mu(z) > 0, B^\mu(l) \in \mathrm d y
	\Big | B^\mu(0) = u \right\}
	P\left\{ 
	B^\mu(t) \in \mathrm d w \Big | B^\mu(l) = y \right\}
}
\\
&=
\frac{
	\frac{y}{l\sqrt{2 \pi l}} e^{ - \frac{y^2}{2l} - \frac{\mu^2l}{2} + \mu y}
}{
	\int_0^\infty 
	\frac{y}{l\sqrt{2 \pi l}} e^{ - \frac{y^2}{2l} - \frac{\mu^2l}{2} + \mu y}
	\frac{e^{ - \frac{(w-y)^2}{2(t-l)} }}{\sqrt{2 \pi (t-l)}} e^{ - \frac{\mu^2}{2}(t-l) + \mu (w-y)}
	\,\mathrm dy
} \,\mathrm d y
\end{align*}
We give an argument to prove that
\begin{align}
	 P\left\{ \Gamma_{l,t}^\mu \in \mathrm d s , B^\mu(t) \in \mathrm dw
	\Big | B^\mu(l) = y \right\}
	=
	 P\left\{ \Gamma_{l,t}^{-\mu} \in \mathrm d s , B^{-\mu}(t) \in \mathrm dy
	\Big | B^{-\mu}(l) = w \right\}
\end{align}

where interchanging the starting point $ y $ with the final
position $ w $ implies the change of sign in the drift. 

This can be intuitively inferred from the conditional Feynman--Kac functional
written as

\begin{align*}
	& \numberthis 
	\mathbb E \left\{
	e^{   - \int_0^t k(B^\mu(s)) \,\mathrm d s }
	\Big |B^\mu (0) = u, B^\mu(t) = 0
	\right\} 
	\\ \qquad 
	&= 
	\lim\limits_{n \to \infty} 
	\idotsint
	e^{  - \sum_{j=1}^n V(x_j)(t_{j+1} - t_j)  }
	P\left\{  \bigcap_{j =1}^n (B^\mu(s_j) \in \mathrm d x_j) 
	\Big | B^\mu (0) = u, B^\mu(t) = 0 \right\}
\end{align*}

for $  0 = t_0 < t_1 < \dots < t_n < t_{n+1} = t $, and then interchanging past and future times.  
From formula \eqref{eq:souj-joint}
\begin{align*}
\numberthis 
	 P\left\{ \Gamma_{l,t}^\mu \in \mathrm d s , B^\mu(t) \in \mathrm d0
	\Big | B^\mu(l) = y \right\}
	&=
	P\left\{ \Gamma_{l,t}^{-\mu} \in \mathrm d s , B^{-\mu}(t) \in \mathrm dy
	\Big | B^{-\mu}(l) = 0 \right\}
	\\
	&=
	\frac{\mathrm d s \,\, \mathrm dy}{2 \pi}
	e^{ - \frac{\mu^2}{2} (t-l) - \mu y} \int_{t-l-s}^{t-l}
	y \frac{e^{ - \frac{y^2}{2(t-l-z)}}}{\sqrt{z^3 (t-l-z)^3} } \mathrm d z
	\\
	&=
	\frac{\mathrm d s \,\, \mathrm dy}{2 \pi}
	e^{ - \frac{\mu^2}{2} (t-l) - \mu y} 
	\int_{0}^{s}
	y \frac{e^{ - \frac{y^2}{2z}}}{\sqrt{z^3 (t-l-z)^3} } \mathrm d z
\end{align*}

In view of all the arguments above the limit of \eqref{eq:souj-exc-dist}
for $ u \to 0 $ yields

\begin{align*}\label{eq:souj-exc-proof}
	& \numberthis 
	\lim\limits_{u\downarrow 0}
	P\left\{ \Gamma_{l,t}^\mu \in \mathrm d s
	\Big | \min_{ 0\leq z \leq l} B^\mu(z) > 0, B^\mu(0) = u, B^\mu(t)  = 0
	\right\} 
	\\
	&=
	\lim\limits_{u\downarrow 0}
	\frac{
		\int_0 ^\infty 
		\frac{1}{\sqrt{2 \pi l}}\left(
		e^{-\frac{(y-u)^2}{2l}} - e^{-\frac{(y+u)^2}{2l}} \right)
		e^{ - \frac{\mu^2 l}{2} + \mu (y-u) }
		e^{ - \frac{\mu^2 (t-l)}{2} - \mu y }
		\int_{0}^{s}
		y \frac{e^{ - \frac{y^2}{2z}}}{\sqrt{z^3 (t-l-z)^3} } \mathrm d z
		\mathrm d y
	}{
		\int_0^\infty 
		\frac{1}{\sqrt{2 \pi l}}\left(
		e^{-\frac{(y-u)^2}{2l}} - e^{-\frac{(y+u)^2}{2l}} \right)
		e^{ - \frac{\mu^2 l}{2} + \mu (y-u) }
		e^{ -\frac{(0-y - \mu (t-l))^2}{2(t-l)} }
		\,
		\frac{\mathrm dy}{\sqrt{2 \pi (t-l)}}
	}
	\frac{ \mathrm d s}{2 \pi}
	\\
	&=
	\sqrt \frac{t-l}{2\pi} 
	\frac{ 
		\int_0^\infty y e^{ - \frac{y^2}{2l}} \,\mathrm d y
		\int_{0}^{s}
		y \frac{e^{ - \frac{y^2}{2z}}}{\sqrt{z^3 (t-l-z)^3} } \mathrm d z
	}{
		\int_0^\infty y e^{ - \frac{y^2}{2l}} 
		 e^{ - \frac{y^2}{2(t-l)}} \,\mathrm d y
	}
	\mathrm d s
	\\
	&=
	\frac{t}{2l \sqrt{t-l}} \int_0^s
	\frac{\mathrm d w}{\sqrt{w^3(t-l-w)^3}} 
	\left(\sqrt{\frac{lw}{l+w}} \right)^3
	\mathrm d s
	\\
	&=
	\frac t2 \sqrt \frac {l}{t-l} 
	\int_0^s
	\frac{\mathrm d w}{\sqrt{(t-l-w)^3(l+w)^3}} 
	\mathrm d s
\\
	&=
\frac t2 \sqrt \frac{l } {t-l} 
\int_{l}^{l+s}
\frac{\mathrm d w} {\sqrt{w^3 (t-w)^3 }}
\mathrm d s
\\
&=
\frac 1t \sqrt \frac{l } {t-l}  
\int_{ \arcsin \sqrt \frac lt }^{ \arcsin \sqrt \frac{l+s}s}
\frac{ \mathrm d \varphi}{\sin^2 \varphi \cos^2 \varphi }
\mathrm d s\\
&=
\frac 1t \sqrt \frac{l } {t-l}
\left\{
\frac{ \sqrt{   l + s }}{  \sqrt{  t - l -s  }} - 
\sqrt \frac{l } {t-l} -
\frac{ \sqrt{  t - l -s  }}{ \sqrt{ l+s}  } + 
\sqrt \frac{t-l}{l}
\right\} \mathrm d s\\
&=
\frac 1t \sqrt \frac{l } {t-l}
\left\{
\frac{t - 2l}{  \sqrt{  l(t - l) }} - 
\frac{t - 2(l + s)}{  \sqrt{  (l+s)(t - l -s) }}  
\right\} \mathrm d s
\qquad 0 < s <t-l . 
\end{align*}

This proves the following theorem. 

\begin{theorem}
	The conditional distribution of $ \Gamma^\mu_{l,t} $ under $ \mu_{0,0}^l $ is given by
	\begin{align*}\label{eq:souj-exc-thm}
	\numberthis
	P\left\{  \Gamma^\mu_{l,t} \in \mathrm d s \Big | \inf_{0< z\leq l} B(z) > 0, B(0)=0 , B(t)=0 \right\} 
	&=
	\frac t2 \sqrt \frac{l } {t-l} 
	\int_{l}^{l+s}
	\frac{\mathrm d w} {\sqrt{w^3 (t-w)^3 }} \,\mathrm d s 
	\\
	&= 
	\frac 1t \sqrt \frac{l } {t-l}
	\left\{
	\frac{t - 2l}{  \sqrt{  l(t - l) }} - 
	\frac{t - 2(l + s)}{  \sqrt{  (l+s)(t - l -s) }}  
	\right\} \mathrm d s  \, .
	\end{align*}
\end{theorem}

\begin{remark}
	We note that the distribution of $ \Gamma^\mu_{l,t} $ in the case of 
	the generalized excursion has not a singular component (differently from the meander).
	Furthermore the distribution \eqref{eq:souj-exc-thm} can be regarded as that of the sojourn time 
	on $ [0, \infty) $ of a Brownian bridge with a Rayleigh distributed starting point
	during an interval of length $ t-l $.
	Moreover, we remark that for
	 $l = \frac t2$ the density \eqref{eq:souj-exc-thm} simplifies as
	\begin{align}\label{eq:souj-cond-d}
	& P\left\{  \Gamma^\mu_{\frac t2,t} \in \mathrm d s \Big |  \inf_{0< z\leq t/2} B(z) > 0, B(0)=0 , B(t)=0 \right\}
	= \\
	&= \frac{ 4s}{ t \sqrt{ t^2 - 4s^2}} \mathrm d s \qquad 0<s<\frac t2 \notag.
	\end{align}
	
	We observe that for $l = 0$ we retrieve in \eqref{eq:souj-exc-thm} the uniform distribution. 
	\vskip 1cm
\end{remark}

%

%
%

\begin{remark}

We are also able to show that the mean value of $\Gamma_{l,t}$
under  $ \mu_{0,0}^l $ is
\begin{align}\label{eq:souj-exp}
&\mathbb E \left( \Gamma_{l,t}^\mu
\Big |  \inf_{0< z\leq l} B(z) > 0, B(0)= B(t)=0 \right )
= \\
&=
\frac t2 \sqrt \frac{l}{t- l}
 \arccos \sqrt \frac lt  + 
\frac{t-2l}{2} \notag 
\end{align}

\begin{proof}
	
In view of an intermediate formula in  \eqref{eq:souj-exc-proof} we have that
\begin{align*}
& \mathbb E \left( \Gamma_{l,t}^\mu
\Big |  \inf_{0< z\leq l} B(z) > 0, B(0)= B(t)=0 \right )
\\
&= \frac t2 \sqrt \frac{l}{t- l}
\int_{0}^{t-l} s
\int_0^s
\frac{\mathrm d w} {\sqrt{(t-l-w)^3 (l+w)^3}}  \,\mathrm d s =\\
& = 
\frac t2 \sqrt \frac{l}{t- l} 
\int_{0}^{t-l}
\frac{\mathrm d w} {\sqrt{(t-l-w)^3 (l+w)^3}} 
\int_{w}^{t-l} s
\,\mathrm d s =\\
& = 
\frac t4 \sqrt \frac{l}{t- l} 
\int_{0}^{t-l}
\frac{(t-l)^2 -w^2} {\sqrt{(t-l-w)^3 (l+w)^3}} \, \mathrm d w
\\
& = 
\frac t4 \sqrt \frac{l}{t- l} 
\int_{0}^{t-l}
\frac{t-l +w} {\sqrt{(t-l-w) (l+w)^3}} \, \mathrm d w
\\
& = 
\frac 12 \sqrt \frac{l}{t- l} 
\int_{\arcsin \sqrt \frac lt}^{\frac \pi 2}
\left\{ \frac{t-2l} {\sin^2 \varphi }  + t\right\}\, \mathrm d \varphi 
\\
&=
\frac t2 \sqrt \frac{l}{t- l}
\left(
\frac \pi 2 - \arcsin \sqrt \frac lt 
\right) + 
\frac{t-2l}{2}
\end{align*}

\end{proof}
For $l=0$ $\mathbb E \Gamma_{0,t}^\mu = \frac t2$, while in the special case where $l = \frac t2$
we have $\mathbb E \Gamma^\mu_{\frac t2,t} = \frac{t\pi}8$.

\end{remark}

\begin{remark}

The distribution function of $\Gamma_{l,t}^\mu$ writes
\begin{align*}
&P\left\{ \Gamma_{l,t}^\mu < \bar{s}
\Big |  \inf_{0 < z\leq l} B^\mu(z) > 0, B^\mu(0)= B^\mu(t)=0 
\right\}  = 
\numberthis \label{eq:souj-dist-fun} \\
&= \frac t2 \sqrt \frac{l}{t- l}
\int_{0}^{\bar s} 
\mathrm d s 
\int_0^s
\frac{\mathrm d w} {\sqrt{(t-l-w)^3 (l+w)^3}}  =\\
& = 
\frac t2 \sqrt \frac{l}{t- l} 
\int_{0}^{\bar s}
\frac{\bar s - w} {\sqrt{(t-l-w)^3 (l+w)^3}} 
\, \mathrm d w\\
&=
\frac t2 \sqrt \frac{l}{t- l} 
\int_{ \arcsin \sqrt \frac lt }^{ \arcsin \sqrt \frac{l+\bar s}t}
\frac{
	\bar s + l - t \sin^2 \varphi
	}{
	\sqrt{(t\sin^2 \varphi)^3 (t\cos^2 \varphi)^3}
	}
2t \sin \varphi \cos\varphi
\, \mathrm d \varphi \\
&=
\frac 1t \sqrt \frac{l}{t- l} 
\int_{ \arcsin \sqrt \frac lt }^{ \arcsin \sqrt \frac{l+\bar s}t}
\left[ \frac{\bar s + l }{\sin^2\varphi \cos^2 \varphi} - 
	\frac{t}{\cos^2\varphi}\right]
\,\mathrm d \varphi
\\
&=
\frac{\bar s + l}t \sqrt \frac{l}{t- l} 
\int_{ \arcsin \sqrt \frac lt }^{ \arcsin \sqrt \frac{l+\bar s}t}
\left[ \frac{\mathrm d}{\mathrm d \varphi} \tan \varphi - 
	\frac{\mathrm d}{\mathrm d \varphi} \cot \varphi \right]
	\,\mathrm d \varphi - 
\sqrt \frac{l}{t- l} 
\int_{ \arcsin \sqrt \frac lt }^{ \arcsin \sqrt \frac{l+\bar s}t}
 \frac{\mathrm d}{\mathrm d \varphi} \tan \varphi 
	\,\mathrm d \varphi
\\
&=	
\frac{\bar s + l}t \sqrt \frac{l}{t- l} 
\left(
	\sqrt \frac{l + \bar s}{t - (l + \bar s)}  - 
	\sqrt \frac{l}{t-l} - 
	\sqrt \frac{t - (l+\bar s)}{l + \bar s} + 
	\sqrt \frac{t-l}{l}
 \right) \\
 & \qquad -
 \sqrt \frac{l}{t- l} 
 \left(
 	\sqrt \frac{l + \bar s}{t - (l +\bar s) } -
 	\sqrt \frac{l}{t - l}
 \right) \\
 &=
 \frac{(\bar s + l) (t - 2l)}{t(t-l)}  + \frac{l}{t-l} - 
 \frac 2t \sqrt  \frac{l}{t-l} \sqrt{(l+\bar s) (t - l - \bar s)}
\end{align*}
for $ 0< \bar s < t-l$ . 
We note that for $ \bar s = t - l  $, this yields the value 1 as expected. 
From \eqref{eq:souj-dist-fun} it emerges that 
\begin{equation*}
\lim\limits_{t\to \infty }
P\left\{  \frac{ \Gamma_{l,t}^\mu } {t} < \bar{s} 
\Big |  \min_{0\leq z\leq l} B^\mu(z) > 0, B^\mu(0)= B^\mu(t)=0 
\right\} 
= \bar s \qquad 0 < \bar s <1
\end{equation*}
that is the random ratio $ \frac{\Gamma_{l,t}^\mu}{t} $ is 
asymptotically uniform on $[0,1]$.

For $ l= \frac t2$ the distribution function becomes
\[ P\left\{ \Gamma^\mu_{\frac t2,t} < \bar{s} \right\}  = 
1 - \frac{ \sqrt{t^2 - 4 \bar s^2}}{t} \qquad 0 \leq \bar s \leq \frac t2 \]

\end{remark}


%
%
%
%
%
%
%
%
%
%

\bibliographystyle{abbrvnat}
\bibliography{biblio}

\begin{thebibliography}{14}
\providecommand{\natexlab}[1]{#1}
\providecommand{\url}[1]{\texttt{#1}}
\expandafter\ifx\csname urlstyle\endcsname\relax
  \providecommand{\doi}[1]{doi: #1}\else
  \providecommand{\doi}{doi: \begingroup \urlstyle{rm}\Url}\fi

\bibitem[Beghin et~al.(2003)Beghin, Nikitin, and Orsingher]{BEGHIN2003291}
L.~Beghin, Y.~Nikitin, and E.~Orsingher.
\newblock How the sojourn time distributions of {Brownian} motion are affected
  by different forms of conditioning.
\newblock \emph{Statistics \& Probability Letters}, 65\penalty0 (4):\penalty0
  291 -- 302, 2003.

\bibitem[Belkin(1970)]{belkin70}
B.~Belkin.
\newblock A limit theorem for conditioned recurrent random walk attracted to a
  stable law.
\newblock \emph{Ann. Math. Statist.}, 41\penalty0 (1):\penalty0 146--163, 02
  1970.

\bibitem[Billingsley(2009)]{billingsley2009convergence}
P.~Billingsley.
\newblock \emph{Convergence of Probability Measures}.
\newblock Wiley Series in Probability and Statistics. Wiley, 2009.
\newblock ISBN 9780470317808.

\bibitem[Bonaccorsi and Zambotti(2004)]{zambotti04}
S.~Bonaccorsi and L.~Zambotti.
\newblock Integration by parts on the {Brownian} meander.
\newblock \emph{Proceedings of the American Mathematical Society}, 132\penalty0
  (3):\penalty0 875--883, 2004.

\bibitem[Chaumont and Doney(2005)]{chaumont2005}
L.~Chaumont and R.~Doney.
\newblock On {Lévy} processes conditioned to stay positive.
\newblock \emph{Electron. J. Probab.}, 10:\penalty0 948--961, 2005.

\bibitem[Chung(1976)]{chung1976}
K.~L. Chung.
\newblock Excursions in {Brownian} motion.
\newblock \emph{Ark. Mat.}, 14\penalty0 (1-2):\penalty0 155--177, 1976.

\bibitem[Denisov(1984)]{denisov84}
I.~V. Denisov.
\newblock A random walk and a {Wiener} process near a maximum.
\newblock \emph{Theor. Prob. Appl.}, 28:\penalty0 821--824, 1984.

\bibitem[Durrett et~al.(1977)Durrett, Iglehart, and Miller]{durrett77}
R.~T. Durrett, D.~L. Iglehart, and D.~R. Miller.
\newblock Weak convergence to {Brownian} meander and {Brownian} excursion.
\newblock \emph{Ann. Probab.}, 5\penalty0 (1):\penalty0 117--129, 1977.

\bibitem[Iafrate and Orsingher(2019)]{iafrate19}
F.~Iafrate and E.~Orsingher.
\newblock Some results on the {Brownian} meander with drift.
\newblock \emph{J. Theor. Probab.}, online since March 14th, 2019.

\bibitem[Iglehart(1974)]{iglehart74}
D.~L. Iglehart.
\newblock Functional central limit theorems for random walks conditioned to
  stay positive.
\newblock \emph{Ann. Probab.}, 2\penalty0 (4):\penalty0 608--619, 08 1974.

\bibitem[It{\^o} and McKean(1996)]{ito1996diffusion}
K.~It{\^o} and H.~McKean.
\newblock \emph{{Diffusion Processes and their Sample Paths}}.
\newblock Classics in Mathematics. Springer Berlin Heidelberg, 1996.

\bibitem[Janson(2007)]{janson07}
S.~Janson.
\newblock {Brownian excursion area, Wright's constants in graph enumeration,
  and other Brownian areas.}
\newblock \emph{Probability Surveys}, 4:\penalty0 80--145, 2007.

\bibitem[Kaigh(1978)]{kaigh78}
W.~Kaigh.
\newblock An elementary derivation of the distribution of the maxima of
  {Brownian} meander and {Brownian} excursion.
\newblock \emph{Rocky Mountain J. Math.}, 8\penalty0 (4):\penalty0 641--646,
  1978.

\bibitem[Karatzas and Shreve(2014)]{karatzas2014brownian}
I.~Karatzas and S.~Shreve.
\newblock \emph{Brownian Motion and Stochastic Calculus}.
\newblock Graduate Texts in Mathematics. Springer New York, 2014.
\newblock ISBN 9781461209492.

\end{thebibliography}

\end{document}